\documentclass[11pt, a4paper]{article}
\usepackage{amssymb,amsmath,latexsym}
\usepackage{graphicx}
\usepackage[T1]{fontenc}
\usepackage{subfigure}
\usepackage{slashbox}
\usepackage{pbox}

\newtheorem{theorem}{Theorem}

\newtheorem{case}{Case}

\newtheorem{corollary}{Corollary}

\newtheorem{lemma}{Lemma}

\newtheorem{proposition}{Proposition}

\newenvironment{proof}{\begin{trivlist} \item[] {\em Proof:}}{\hfill $\Box$
                       \end{trivlist}}

\title{On Perfect and Quasiperfect Dominations in Graphs\thanks{Research supported by projects ESF EUROCORES programme EUROGIGA-ComPoSe IP04-MICINN, MTM2012-30951/FEDER, MTM2011-28800-C02-01, Gen. Cat. DGR 2014SGR46, Gen. Cat. DGR 2009SGR1387 and JA-FQM 305.}}
\author{José Cáceres$^1$ \and Carmen Hernando$^2$ \and Mercè Mora$^2$ \and Ignacio M. Pelayo$^2$ \and María Luz Puertas$^1$}
\date{}
\begin{document}
\maketitle

\footnotetext[1]{Universidad de Almería, Almería, Spain}
\footnotetext[2]{Universitat Politècnica de Catalunya, Barcelona, Spain}

\begin{abstract}

A subset $S\subseteq V$ in a graph $G=(V,E)$ is a $k$-quasiperfect dominating set (for $k\geq 1$) if every vertex not in $S$ is adjacent to at least one and at most $k$ vertices in $S$. The cardinality of a minimum $k$-quasiperfect dominating set in $G$ is denoted by $\gamma_ {\stackrel{}{1k}}(G)$. Those sets were first introduced by Chellali et al. (2013) as a generalization of the perfect domination concept and allow us to construct a decreasing chain of quasiperfect dominating numbers $ n \ge \gamma_ {\stackrel{}{11}}(G) \ge \gamma_ {\stackrel{}{12}}(G)\ge \ldots \ge \gamma_ {\stackrel{}{1\Delta}}(G)=\gamma(G)$ in order to indicate how far is $G$ from being perfectly dominated. In this paper we study properties, existence and realization of graphs for which the chain is short, that is, $\gamma_ {\stackrel{}{12}}(G)=\gamma (G)$. Among them, one can find cographs, claw-free graphs and graphs with extremal values of $\Delta(G)$.
\end{abstract}
\maketitle

\section{Introduction}

All the graphs considered here are  finite, undirected, simple, and connected. Given a graph $G=(V,E)$, the \emph{open neighborhood} of $v\in V$ is $N(v)=\{u\in V; uv\in E\}$ and the \emph{closed neighborhood} is $N[v]=N(v)\cup\{v\}$. The \emph{degree} $\deg(v)$ of a vertex $v$ is the number of neighbors of $v$, i.e., $\deg(v)=|N(v)|$. The \emph{maximum degree} of $G$, denoted by $\Delta(G)$, is the largest degree among all vertices of $G$. Similarly, it is defined the \emph{minimum degree} $\delta(G)$. If a vertex is adjacent to any other vertex then it is called \emph{universal}. For undefined basic concepts we refer the reader to introductory graph theoretical literature as~\cite{chlepi11}.

Given a graph $G$, a subset $S$ of its vertices is a \emph{dominating set} of $G$ if every vertex $v$ not in $S$ is adjacent to at least one vertex in $S$, or in other words $N(v)\cap S\neq\emptyset$. The \emph{domination number} $\gamma(G)$ is the minimum cardinality of a dominating set of $G$, and a dominating set of cardinality $\gamma(G)$ is called a \emph{$\gamma$-code}~\cite{hahesl}.

The most efficient way for a set $S$ to dominate occurs when every vertex not in $S$ is adjacent to exactly one vertex in $S$. In that case, $S$ is called a \emph{perfect dominating set}, which were introduced in~\cite{cohahela} and studied in~\cite{biggs,bode96,cahahe12,dd09,dw93,feho91,kwl14,list90} under different names. We denote by $\gamma_ {\stackrel{}{11}}(G)$ the minimum cardinality of a perfect dominating set of $G$ and called it the \emph{perfect domination number}. A perfect dominating set of cardinality $\gamma_ {\stackrel{}{11}}(G)$ is called a \emph{$\gamma_ {\stackrel{}{11}}$-code}.

Not always is possible to achieve perfection, so it is natural to wonder if we can obtain something close to it. In~\cite{chhahemc13}, the authors defined a generalization of perfect dominating sets called a \emph{$k$-quasiperfect dominating set} for $k\geq 1$ (\emph{$\gamma_ {\stackrel{}{1k}}$-set} for short). Such a set $S$ is a dominating set where every vertex not in $S$ is adjacent to at most $k$ vertices of $S$ (see also~\cite{dejter09, yang}). The \emph{$k$-quasiperfect domination number} $\gamma_ {\stackrel{}{1k}}(G)$ is the minimum cardinality of a $\gamma_ {\stackrel{}{1k}}$-set of $G$ and a \emph{$\gamma_ {\stackrel{}{1k}}$-code} is a $\gamma_ {\stackrel{}{1k}}$-set of cardinality $\gamma_ {\stackrel{}{1k}}(G)$. Certainly, $\gamma_ {\stackrel{}{11}}$-sets and  $\gamma_ {\stackrel{}{1\Delta}}$-sets are respectively perfect dominating and dominating sets. Thus, given a graph $G$ of order $n$ and maximum degree $\Delta$, one can construct the following decreasing chain of quasiperfect domination parameters:

$$ n \ge \gamma_ {\stackrel{}{11}}(G) \ge \gamma_ {\stackrel{}{12}}(G)\ge \ldots \ge \gamma_ {\stackrel{}{1\Delta}}(G)=\gamma(G) $$

For any graph $G$, the values in this chain give us an idea about how far is $G$ from being perfectly dominated. Particularly, in this work we focus our attention when the chain is \emph{short}, or in other words $\gamma_ {\stackrel{}{12}}(G) =\gamma(G)$. The next result, obtained in~\cite{chhahemc13}, provides a variety of families for which the chain is short.

\begin{theorem}\label{condicion}
If $G$ is a graph of order $n$ verifying at least one of the following conditions:
\begin{enumerate}
\item $\Delta(G)\geq n-3$;
\item $\Delta(G)\leq 2$;
\item $G$ is a cograph;
\item $G$ is a claw-free graph;
\end{enumerate}
then $\gamma_ {\stackrel{}{12}}(G) =\gamma(G)$.
\end{theorem}

The paper is organized as follows: in the next section we introduce several well-known and technical results that will be useful for the rest of the paper. The next two sections deal with the study of the cases of Theorem~\ref{conditions}, thus  Section~\ref{extremal} is devoted to the extremal degree families and Section~\ref{freegraphs} to cographs and claw-free graphs.

\section{Basic and general results}\label{basics}

In this section, we review some results founded in the literature about quasiperfect parameters as well as introduce some basic technical results that will be useful in the rest of the paper. The next table summarizes the values of parameters under consideration for some simple families of graphs:

\begin{table}[htbp]
\begin{center}
\begin{tabular}{ccccccc}
   & paths & cycles & cliques & stars & bicliques & wheels \\
\hline\noalign{\smallskip}
$G$  & $P_n$  & $C_{n}$ & $K_n$ & $K_{1,n-1}$ & $K_{p,n-p}$ &$W_{n}$  \\
\noalign{\smallskip}\hline\noalign{\smallskip}
 $n$ & $n\geq 3$ & $n\geq 4$ & $n\geq 2$ & $n\geq 4$ & $2\leq p\leq n-p$ & $n\geq 3$\\
\noalign{\smallskip}\hline\noalign{\smallskip}
$\Delta(G)$ & $2$ & $2$ & $n-1$ & $n-1$ & $n-p$  &  $n-1$ \\\hline
$\gamma_ {\stackrel{}{11}}(G)$ & $\lceil\frac{n}{3}\rceil$ & $\lceil\frac{2n}{3}\rceil-\lfloor\frac{n}{3}\rfloor$ & $1$ & $1$ & $2$  &  $1$ \\
$\gamma_ {\stackrel{}{12}}(G)$ & $\lceil\frac{n}{3}\rceil$ & $\lceil\frac{n}{3}\rceil$ & $1$ & $1$ & $2$  &  $1$ \\
$\gamma(G)$    &      $\lceil\frac{n}{3}\rceil$     & $\lceil\frac{n}{3}\rceil$ & $1$ & $1$ & $2$ & $1$ \\
\noalign{\smallskip}\hline\noalign{\smallskip}
\end{tabular}
\caption{Quasiperfect domination parameters of some basic graphs.}
\label{tab:ghs}
\end{center}
\end{table}

The next result provides a number of basic technical facts:
\begin{proposition}\label{technical}
Let $G=(V,E)$ a graph of order $n$. In the following, $\Delta(G)=\Delta$, $\gamma(G)=\gamma$, $\delta(G)=\delta$ and let $k$ and $r$ be two positive integers such that $1\leq k\leq \Delta$ and $k\leq r\leq n$:
\begin{enumerate}
\item If $\gamma\leq \Delta$, then $\gamma_ {\stackrel{}{1\gamma}}(G)=\ldots=\gamma_ {\stackrel{}{1\Delta}}(G)=\gamma$;
\item $\gamma_ {\stackrel{}{1\delta}}(G)<n$;
\item \label{three} $\gamma_ {\stackrel{}{11}}(G)=1$ if and only if $\Delta=n-1$.
\item $\gamma_ {\stackrel{}{11}}(G)\leq n-\ell(G)$ where $\ell(G)$ is the number of vertices of degree one.
\item Let $S$ be a $\gamma_ {\stackrel{}{1k}}$-set of $G$ and $v \in V$. If $|N(v)\cap S| >k$ then $v \in S$.
\item Let $S$ be a $\gamma_ {\stackrel{}{1k}}$-set of $G$ and let $K$ be a clique of $G$. If $|V(K)\cap S| > k$ then $V(K)\subseteq S$.
\end{enumerate}
\end{proposition}

\begin{proof}
\begin{enumerate}
\item Assume $S$ is a $\gamma$-code and let $i$ be any integer such that $\gamma \leq i\leq \Delta$. Then, for any vertex $v$ not in $S$ it is clear that $1 \le |N(v) \cap S| \le \gamma$. Hence, $S$ is an $\gamma_ {\stackrel{}{1i}}$-set, and consequently $\gamma_ {\stackrel{}{1i}}(G)\le \gamma$.
\item Now let $v\in V$ with $\deg(v)=\delta$. Since $1\le|N(v)\cap S|\leq |N(v)|=\delta$, the set $S=V\setminus \{v\}$ is a $\gamma_ {\stackrel{}{1\delta}}$-set and consequently, $\gamma_ {\stackrel{}{1\delta}}(G)\le n-1<n$.
\item Assume that $S=\{v\}$ is a $\gamma_ {\stackrel{}{11}}$-code of $G$. Then $v$ is a universal vertex, i.e., $\deg(v)=n-1$. Conversely, let $v\in V$ with $\deg(v)=n-1$. Since $N(v)=V\setminus \{v\}$, the set $\{v\}$ is a $\gamma_ {\stackrel{}{11}}$-code.
\item Let $u$ be a vertex with a unique neighbor $v$. If $S$ is a $\gamma_ {\stackrel{}{11}}$-code containing $u$, then $(S\setminus \{u\})\cup \{v\}$ is also a $\gamma_ {\stackrel{}{11}}$-code.
\item If $S$ is a $\gamma_ {\stackrel{}{1k}}$-set, then no vertex outside $S$ can have more than $k$ neighbors in $S$. So if a vertex $v$ verifies $N|(v)\cap S|>k$ then it belongs to $S$.
\item Similarly as the previous case, if there exists a clique $K$ in $G$ with more than $k$ vertices in $S$, then any vertex in $V(K)\setminus S$, if exists, has more than $k$ neighbors in $S$ so they should be in $S$ or $S$ cannot be a $\gamma_ {\stackrel{}{1k}}$-set. Consequently $V(K)\subseteq S$.
\end{enumerate}
And the proof is complete.
\end{proof}

From the computational point of view, it is important that $k$-quasiperfect domination numbers can be expressed in terms of an integer program. The formulation is as follows: given a vertex subset $S\subseteq V(G)=\{v_1,\ldots ,v_n\}$ the characteristic column vector $X_S=(x_i)$ for $1\leq i\leq n$ satisfies $x_i=1$ if $v_i\in S$ and $x_i=0$ otherwise. Now $S$ is a dominating set if $|N(v_i)\cap S|\geq 1$ for $v_i\notin S$ or, equivalently, if $N\cdot X_S\geq 1_n-X$ where $1_n$ is an $n$ column vector with all its components equal to one (see~\cite{hahesl}). The set $S$ is a $k$-quasiperfect dominating set if it is dominating and $|N(v_i)\cap S|\leq k$ for $v_i\notin S$. Since clearly, $|N(v_i)\cap S|\leq n-1$ for $v_i\in S$, this two conditions can be expressed as $N\cdot X\leq k1_n + (n-k-1)X$. Thus the final formulation for the $k$-quasiperfect domination number $\gamma_ {\stackrel{}{1k}}(G)$:
$$
\begin{aligned}
\gamma_ {\stackrel{}{1k}}(G)=&\min \sum_{i=1}^n x_i\\
\textrm{subject to } & N\cdot X\geq 1_n-X \\
& N\cdot X\leq k1_n + (n-k-1)X\\
& \textrm{with } x_i\in\{0,1\}
\end{aligned}
$$

\section{Extremal degree families}\label{extremal}
Extremal values of the maximum degree $\Delta(G)$ leads to a short quasiperfect domination chain as it was stated in Theorem~\ref{condicion}. In this section, we examine the relationship between extremal values of the maximum degree and the quasiperfect domination parameters. Note that if $\Delta(G)\leq 2$, then $G$ is claw-free which will be considered in Subsection~\ref{clawfreeg}. On the other hand, $\Delta(G)=n-1$ only for graphs with a universal vertex. Hence, this section is divided into the following remaining extremal cases: $\Delta(G)=n-2$, $\Delta(G)=n-3$ and $\Delta(G)=3$.

\begin{figure}[htbp]
  \begin{centering}
    \subfigure[$\gamma_ {\stackrel{}{11}}(G)=2$.\label{2a}]{\includegraphics[width=0.17\textwidth]{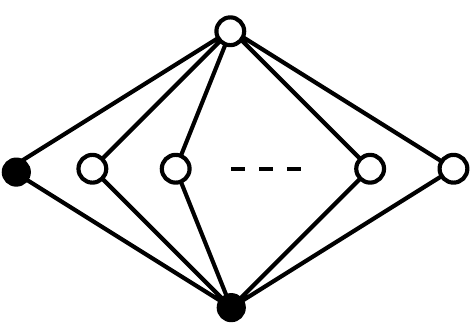}} \hspace{0.25cm}
    \subfigure[$\gamma_ {\stackrel{}{11}}(G)=3$.\label{2b}]{\includegraphics[width=0.15\textwidth]{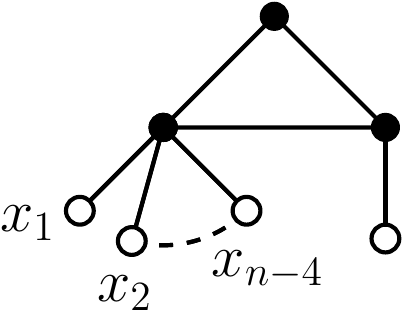}} \hspace{0.25cm}
    \subfigure[$\gamma_ {\stackrel{}{11}}(G)=n$.\label{2c}]{\includegraphics[width=0.17\textwidth]{{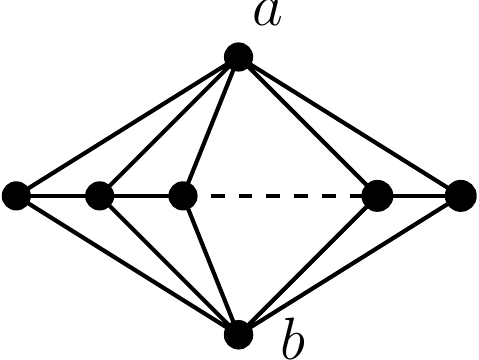}}}\hspace{0.25cm}
    \subfigure[$\gamma_ {\stackrel{}{11}}(G)=n-1$.\label{2d}]{\includegraphics[width=0.17\textwidth]{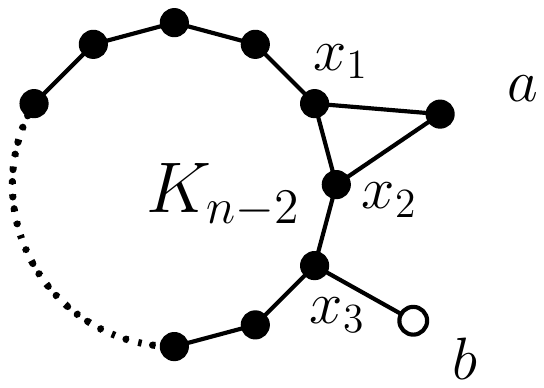}} \hspace{0.25cm}
    \subfigure[$n-\gamma_ {\stackrel{}{11}}(G)\geq 2$.\label{2e}]{\includegraphics[width=0.17\textwidth]{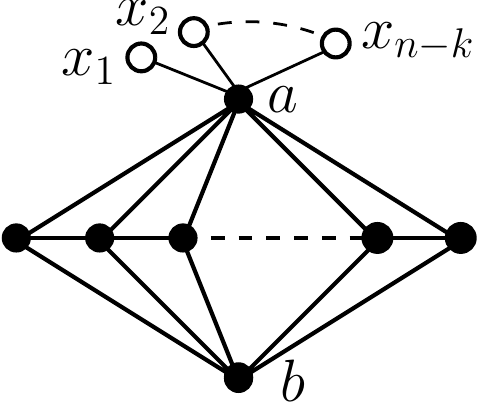}}
\caption{Graphs of order $n$, maximum degree $\Delta =n-2$ and any possible value $k$ of $\gamma_ {\stackrel{}{11}}$.}\label{figdeltan-2}
  \end{centering}
\end{figure}

\subsection{$\Delta(G)=n-2$}\label{deltan-2}
As it was pointed out in Theorem~\ref{condicion}, if $\Delta(G)=n-2$ then $\gamma_ {\stackrel{}{12}}(G)=\gamma(G)$ and in this case $\gamma(G)=2$. Now let us see whether or not there exists a graph under these conditions with any value of $\gamma_ {\stackrel{}{11}}(G)$. The next result answers this question:

\begin{theorem}
Let $k,n$ be positive integers such that $n\ge 4$ and $2\leq k\leq n$. Then, there exists a graph $G$ of order $n$ such that  $\Delta(G)=n-2$ and $\gamma_ {\stackrel{}{11}}(G)=k$ if and only if $(k,n)\notin \{(3,4),(4,4),(4,5),(5,5)\}$.
\end{theorem}
\begin{proof}
The only graphs of order $4$ and $\Delta(G)=2$ are the cycle $C_4$ and the path $P_4$, and in both cases $\gamma_ {\stackrel{}{11}}(G)=2$. On the other hand, there are eight graphs of order $5$ and maximum degree $\Delta(G)=3$, all of them having either $\gamma_ {\stackrel{}{11}}(G)=2$ or $\gamma_ {\stackrel{}{11}}(G)=3$.

Thus, the only remaining case is when $n\geq6$.

When $k=2$, the graph $G=K_{2,n-2}$ (see Figure~\ref{2a}) has $\Delta(G)=n-2$ and $\gamma_ {\stackrel{}{11}}(G)=2$ since the black vertices form a $\gamma_ {\stackrel{}{11}}$-code.

For the case $k=3$, we consider the graph in Figure~\ref{2b} where the three black vertices are a $\gamma_ {\stackrel{}{11}}$-code.

For the case $k=n$, we construct the graph $P_{n-2}\vee \overline{K_2}$ showed in Figure~\ref{2c}. Let $V(P_{n-2})=\{1,\dots, n-2 \}$ and $V(\overline{K}_2)=\{a,b \}$. Obviously, $\Delta(G)=n-2$. Let us see that $\gamma_ {\stackrel{}{11}}(G)=n$, that is, the unique $\gamma_ {\stackrel{}{11}}$-set is $V(G)$.
Let $S$ be a $\gamma_ {\stackrel{}{11}}$-set of $G$, then $\vert S \vert\geq2$. We distinguish tree cases:
\begin{itemize}
  \item  If $\{a,b \} \subset S$ then $\vert N(i)\cap S \vert >1$ for all $i \in V(P_{n-2})$. By Proposition~\ref{technical}, $V(P_{n-2}) \subset S$, that is, $S=V(G)$.
  \item  If $\{i,j \} \subset S$, for some $\{i,j \}\subset V(P_{n-2})$, then $\vert N(a)\cap S \vert >1$ and $\vert N(b)\cap S \vert >1$. Then $ \{a,b \} \subset S$. By previous item, then  $S=V(G)$.
  \item  If $\{a,i \} \subset S$ for some $i \in P_{n-2}$, then $\vert N(i+1)\cap S \vert >1$ or $\vert N(i-1)\cap S \vert >1$. In any case $\vert V(P_{n-2}) \cap S \vert>1$, and then, by previous item, $S=V(G)$.
\end{itemize}

The following case occurs when $k=n-1$. We construct the graph in Figure~\ref{2d}. Note that the set of black vertices is a $\gamma_ {\stackrel{}{11}}$-set with cardinality $n-1$. We claim that any $\gamma_ {\stackrel{}{11}}$-code should contain those vertices. Let $S$ be a $\gamma_ {\stackrel{}{11}}$-code. Since $S$ is dominating, $S\cap \{x_1,x_2,a\}\neq \emptyset$ and $S\cap \{x_3,b\}\neq \emptyset$. Observe that $\{a,b\}$ is not a dominating set, and if $S$ contains two vertices of $V(K_{n-2})$ then it contains all the vertices of the clique plus the vertex $a$. Thus, it only remains two check the cases $\{x_1,b\}\subseteq S$ (the case $\{x_2,b\}\subseteq S$ is analogous) and $\{x_3,a\}\subseteq S$. However in the former case, we have that $x_3\in S$ and in the later, $x_1\in S$. Hence in any case, there are two vertices of the clique in $S$ and consequently all the black vertices belong to $S$.

Finally, suppose $n-k\geq 2$ and $k\geq 4$. The graph we consider is the one depicted in Figure~\ref{2e}. We denote $W=V(P_{k-2}\vee \overline{K_2})$ and $V(\overline{K_2})=\{a,b \}$. Obviously $\Delta(G)=n-2$, $\{a,b\}$ is a $\gamma$-code and also a $\gamma_ {\stackrel{}{1k}}$-code for $k\geq 2$. On the other hand, $W$ is a $\gamma_ {\stackrel{}{11}}$-set of $G$ and $|W|=k$. It only remains to prove that any $\gamma_ {\stackrel{}{11}}$-set of $G$ contains $W$.

Let $S$ be a $\gamma_ {\stackrel{}{11}}$-set of $G$, then its cardinality is at least $2$. As $S$ dominates all the vertices, $\vert W \cap S \vert\geq1$. If  $\vert W \cap S \vert\geq2$ then, analogously as an above case, we obtain $W\subset S$. We suppose that $\vert W \cap S \vert=1$. In this case, as $V \setminus W$ does not dominate $b$, then it is necessary that $W \cap S\subset N[b]$. Therefore, $a\notin S$ and  $\{x_1, \dots, x_{n-k}\}\subset S$ because $S$ is a dominating set. But in this case, $\vert N(a)\cap S \vert>1$, that is a contradiction. Finally, we conclude that $\gamma_ {\stackrel{}{11}}(G)=\vert W \vert=k$.
\end{proof}

\begin{figure}[htbp]
  \begin{centering}
    \subfigure[$P_5$ and $C_5$.\label{c5}]{\includegraphics[width=0.16\textwidth]{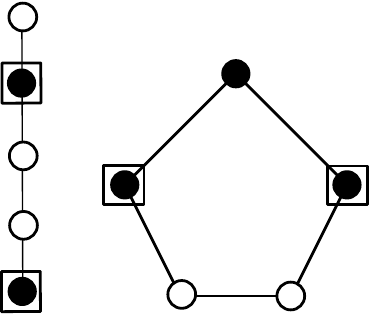}} \hspace{1.25cm}
    \subfigure[$|V(G)|=n\geq 6$, $\gamma =\gamma_ {\stackrel{}{11}}=2$.\label{22n}]{\includegraphics[width=0.22\textwidth]{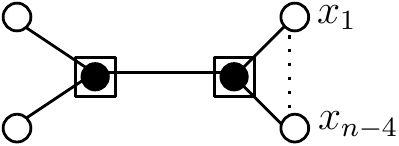}} \hspace{1.25cm}
    \subfigure[$|V(G)|=n\geq 6$, $\gamma =2$, $\gamma_ {\stackrel{}{11}}=3$.\label{23n}]{\includegraphics[width=0.22\textwidth]{{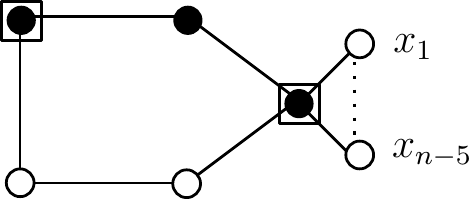}}}
\caption{Some graphs with order $n\geq 5$ and maximum degree $n-3$. The set of squared vertices in each graph is a $\gamma$-code whereas black vertices form a $\gamma_ {\stackrel{}{11}}$-code.}\label{22nand23n}
  \end{centering}
\end{figure}

\subsection{$\Delta(G)=n-3$}\label{deltan-3}
For this case and as in the previous subsection, by Theorem~\ref{condicion} it holds $\gamma_ {\stackrel{}{12}}(G)=\gamma(G)$. However, either $\gamma(G)=2$ or $3$, and for both cases we characterize whether or not there exists a graph for any value of $\gamma_ {\stackrel{}{11}}(G)$.

\noindent
\begin{table}[htbp]
\begin{tabular}{|c|c|c|c|c|c|}
\hline
\backslashbox{$\gamma_ {\stackrel{}{11}}$}{n} & 7 & 8 & 9 & 10 & $\geq$11 \\
\hline
4 & \multicolumn{5}{l|}{\includegraphics[width=2.5cm]{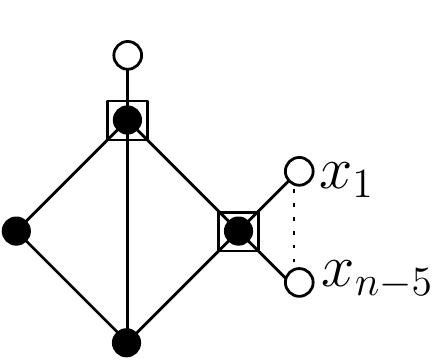}  \ \ \ $n\geq 7$  }\\
\hline
5 & \includegraphics[width=2.2cm]{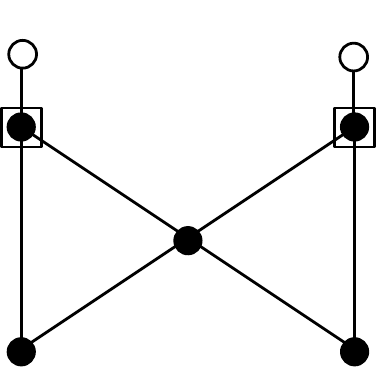} &\includegraphics[width=2.5cm]{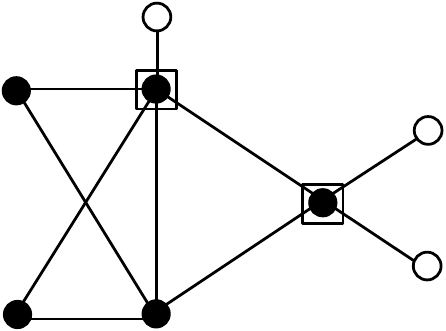} &  \multicolumn{3}{l|}{ \includegraphics[width=2.5cm]{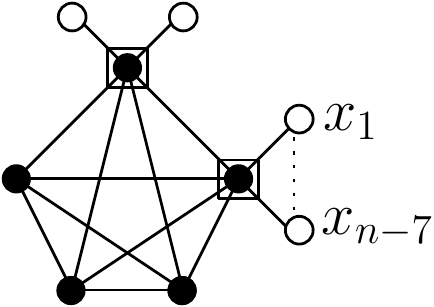} \ \ $n\geq 9$}\\
\hline
6 &\includegraphics[width=2.5cm]{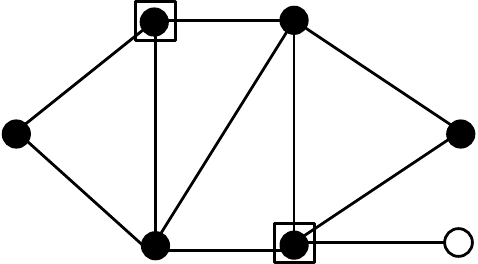} &\includegraphics[width=2.2cm]{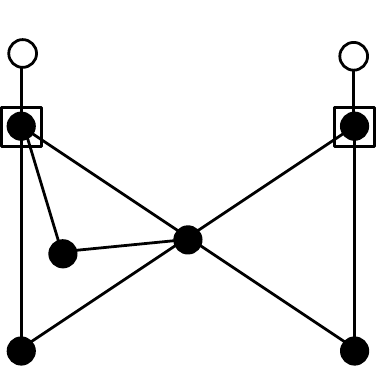} & \includegraphics[width=2.5cm]{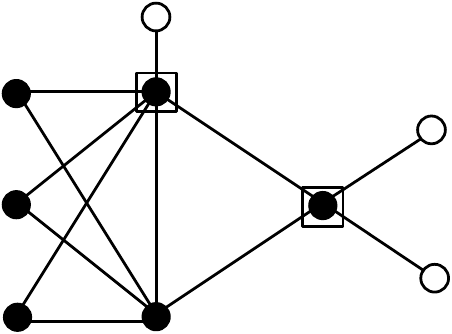} & \multicolumn{2}{l|}{\includegraphics[width=2.5cm]{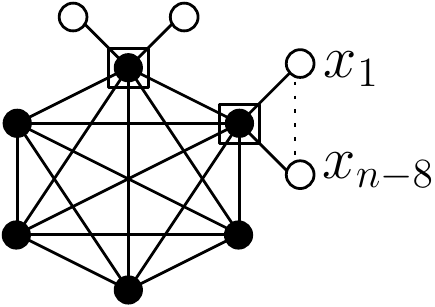} \ \ \ $n\geq 10$}\\
\hline
7& \includegraphics[width=2.5cm]{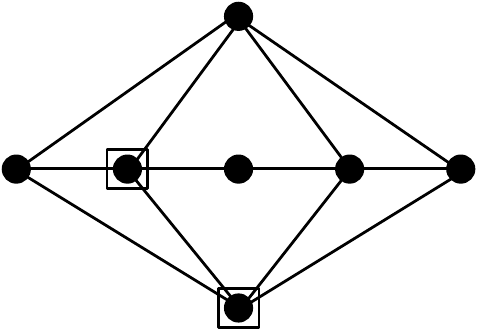}  &\includegraphics[width=2.5cm]{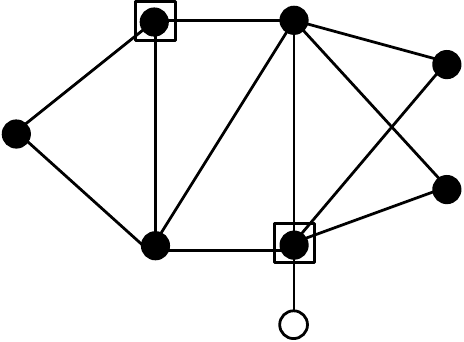}  & \includegraphics[width=2.5cm]{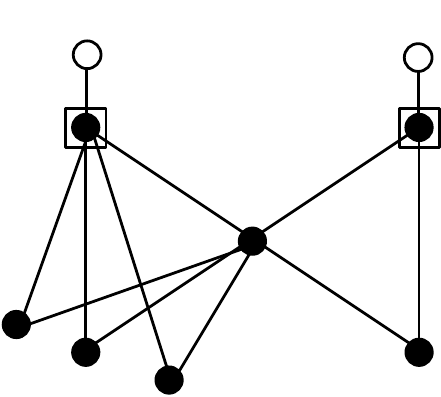}    & \includegraphics[width=2.5cm]{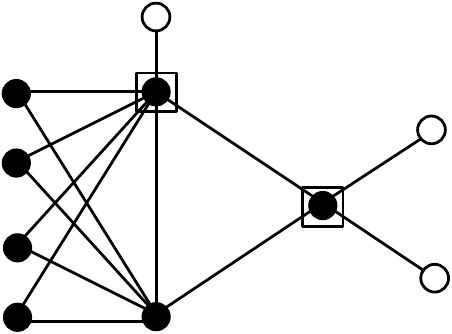}  & \includegraphics[width=2.5cm]{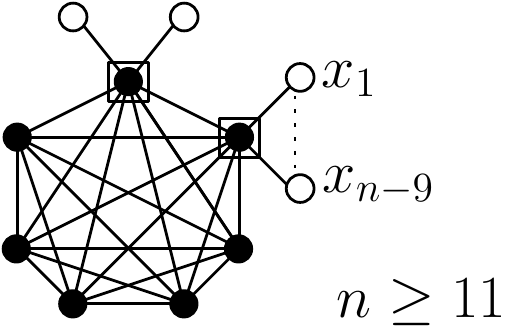}\\
\hline
$\geq$8 & \pbox{20cm}{no sense\\  \\  } & \multicolumn{4}{l|}{ \includegraphics[width=5cm]{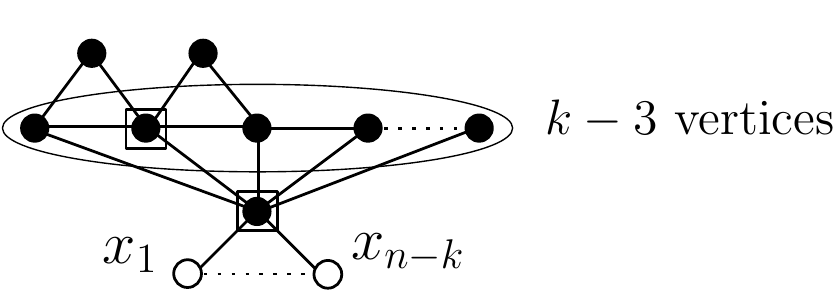} \ \ \ $8\leq \gamma_ {\stackrel{}{11}}=k\leq n$} \\
\hline
\end{tabular}
\caption{Examples of graphs with $n\geq 7$, $\Delta(G)=n-3$, $\gamma=2$ and $\gamma_ {\stackrel{}{11}}\geq 4$. The pair of squared vertices in each graph is a $\gamma$-code and black vertices form a $\gamma_ {\stackrel{}{11}}$-code. Observe that every $\gamma_ {\stackrel{}{11}}$-code is formed by all non-leaf vertices.}
\label{tablegamma2}
\end{table}

\begin{theorem}
Let $(k,n)$ be a pair of integers such that $2\leq k\leq n$ and $n\geq 5$. Then, there exists a graph $G$ such that $|V(G)|=n$, $\Delta (G)=n-3$, $\gamma (G)=2$ and $\gamma_ {\stackrel{}{11}}(G)=k$, if and only if $(k,n)\notin \{(4,5),(5,5),(4,6),(5,6),(6,6)\}$.
\end{theorem}

\begin{proof}
The unique graphs with 5 vertices that satisfy $\Delta(G)=5-3=2$ are the path $P_5$ and the cycle $C_5$. In Figure~\ref{c5}, it is straightforward to check that $\gamma(P_5)=\gamma_ {\stackrel{}{11}}(P_5)=2$ (squared black vertices) and $\gamma(C_5)=2$ (squared vertices), $\gamma_ {\stackrel{}{11}}(C_5)=3$ (black vertices).

Graphs in Figure~\ref{22n} satisfy $n\geq 6$, $\Delta(G)=n-3$, $\gamma =\gamma_ {\stackrel{}{11}}=2$, and squared black vertices are both a $\gamma$-code and a $\gamma_ {\stackrel{}{11}}$-code. On the other hand, graphs in Figure~\ref{23n} satisfy $n\geq 6$, $\Delta(G)=n-3$, $\gamma =2$ and $\gamma_ {\stackrel{}{11}}=3$, where the pair of squared vertices is a $\gamma$-code and the set of black vertices is a $\gamma_ {\stackrel{}{11}}$-code. Note that there is no graph $G$ satisfying $n=6$, $\Delta(G)=6-3=3$ and $\gamma_ {\stackrel{}{11}}\geq 4$ (see Theorem~\ref{caseDelta=3} in the next subsection).

Finally, in Table~\ref{tablegamma2}, some examples of graphs with $n\geq 7$, $\Delta(G)=n-3$, $\gamma=2$ and $\gamma_ {\stackrel{}{11}}\geq 4$ are shown. It is easy to verify that the sets given in this table are $\gamma$-codes or $\gamma_ {\stackrel{}{11}}$-codes using the Proposition~\ref{technical}.
\end{proof}

\begin{figure}[htbp]
\begin{centering}
    \subfigure[$|V(G)|=n\geq 6$, \newline $\gamma=\gamma_ {\stackrel{}{11}}=3$.\label{33n}]{\includegraphics[width=0.18\textwidth]{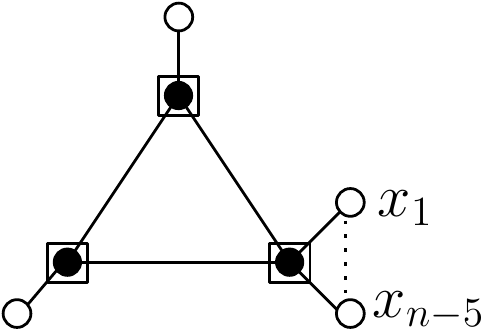}} \hspace{1.5cm}
    \subfigure[$|V(G)|=n\geq 7$, \newline $\gamma=3$ and $\gamma_ {\stackrel{}{11}}=4$.\label{34n}]{\includegraphics[width=0.18\textwidth]{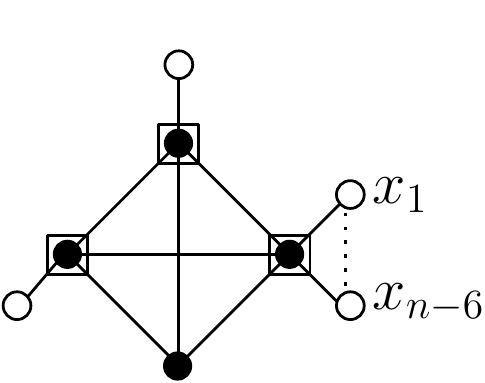}}
\caption{Small cases for Theorem~\ref{th:gamma3}.\label{fig:gamma3}}
\end{centering}
\end{figure}

\begin{theorem}\label{th:gamma3}
Let $(k,n)$ be a pair of integers such that $3\leq k\leq n$ and $n\geq 6$. Then, there exists a graph $G$ such that $|V(G)|=n$, $\Delta (G)=n-3$, $\gamma (G)=3$ and
$\gamma_ {\stackrel{}{11}}(G)=k$, if and only if $(k,n)\notin\{(4,6),(5,6),(6,6),(5,7),(6,7),(7,7),(8,8)\}$.
\end{theorem}

\begin{proof}
Graphs in Figure~\ref{33n} satisfy $n\geq 6$, $\Delta(G)=n-3$, $\gamma =\gamma_ {\stackrel{}{11}}=3$, and squared black vertices are both a $\gamma$-code and a $\gamma_ {\stackrel{}{11}}$-code. In Figure~\ref{34n}, we have an example of graphs with $n\geq 7$, $\Delta(G)=n-3$, $\gamma =3$ and $\gamma_ {\stackrel{}{11}}=4$. Note that there is no graph $G$ satisfying $n=6$, $\Delta(G)=6-3=3$ and $\gamma_ {\stackrel{}{11}}\geq 4$ as we will see in the next subsection.
There exist 16 non-isomorphic graphs with 7 vertices, maximum degree 4, domination number 3 and at most 2 vertices of degree 1 (see~\cite{sage}), and 46 non-isomorphic graphs with 8 vertices, maximum degree 5, domination number 3 and with no vertices of degree 1 (see~\cite{sage}). By inspection, we have checked that there is no case belonging to $\{(5,7),(6,7),(7,7),(8,8)\}$.

Finally, we show examples of graphs with $n\geq 7$, $\Delta(G)=n-3$, $\gamma=3$ and $\gamma_ {\stackrel{}{11}}=4$ in Figure~\ref{fig:gamma3} and with $\gamma_ {\stackrel{}{11}}=k\geq 5$ in Table~\ref{tablegamma3}.
\end{proof}

\subsection{$\Delta(G)=3$}\label{delta3}

Note that if $\Delta(G)\leq 2$, then the graph $G$ is claw-free which will be studied in Subsection~\ref{clawfreeg}. Hence, here we focus our attention on $\Delta(G)=3$.

In~\cite{chhahemc13} it is shown that $\Delta(G)\leq 4$ implies $\gamma_ {\stackrel{}{12}}(G)\leq n-1$. We use similar techniques in the case $\Delta(G)=3$ to prove that $\gamma_ {\stackrel{}{11}}(G)\leq n-3$. Observe that this case does not appear in Theorem~\ref{condicion}, but the results of this subsection are useful for some proofs of previous subsections.

\begin{lemma}\label{conditions}
Let $G$ be a graph of order $n$ and $\Delta (G)=3$ such that at least one of the following conditions holds:
\begin{enumerate}
\item \label{c1} If $G$ contains an induced cycle $C$ such that all of its vertices have degree 3, then $\gamma_ {\stackrel{}{11}}\leq n-|V(C)|\leq n-3$.
\item \label{c2} If there exist two vertices $u,v \in V(G)$ with $\deg(u)=\deg(v)=2$, $d(u,v)\geq 2$ and there is an induced path $P$ joining them such that all of its vertices, other than $u$ and $v$, have degree 3, then $\gamma_ {\stackrel{}{11}}\leq n-|V(P)|\leq n-3$.
\end{enumerate}
\end{lemma}
\begin{proof}
If $G$ satisfies condition 1, the set $V(G)\setminus V(C)$ is a $\gamma_ {\stackrel{}{11}}$-set, and if $G$ satisfies condition 2, the set $V(G)\setminus V(P)$ is a $\gamma_ {\stackrel{}{11}}$-set.
\end{proof}

\begin{theorem}\label{caseDelta=3}
Let $G$ be a graph of order $n$ and $\Delta (G)=3$, other than the bull graph. Then $\gamma_{11}(G)\leq n-3$. Moreover the bull graph $B$ satisfies $\gamma_{11}(B)=3=n-2$.
\end{theorem}

\begin{proof}
Note that vertices belonging to the triangle of the bull graph $B$ are a $\gamma_ {\stackrel{}{11}}$-code, and hence  $\gamma_ {\stackrel{}{11}}(B)=3$ (see Figure~\ref{bull}).

Suppose next that $G$ is a tree. It is clear that $\Delta(G)=3$ implies that there are at least three leaves, so by Proposition~\ref{technical} $\gamma_ {\stackrel{}{11}}\leq n-3$. So we may assume that $G$ (strictly) contains at least one induced cycle $C$. We consider different situations regarding this cycle.

\begin{case}
The cycle $C$ contains two vertices $a$ and $b$ with $\deg(a)=\deg(b)=2$, $d(a,b)\geq 2$.
\end{case}
It is clear that $C$ contains at least a vertex $w$ of degree 3, so walking from $w$ along the cycle towards both directions, we find two vertices $u$ and $v$ in $C$, such that $\deg(u)=\deg(v)=2$, $d(u,v)\geq 2$ and all vertices, other than $u$ and $v$, in the induced path on the cycle between them, have degree 3. So $G$ satisfies condition 2 of Lemma~\ref{conditions}.

\noindent
\begin{table}[htbp]
\begin{center}
\begin{tabular}{|c|c|c|c|c|c|}
\hline
\backslashbox{$\gamma_ {\stackrel{}{11}}$}{n} & 8 & 9 & 10 & $\geq$11 \\
\hline
5 &  \multicolumn{4}{l|}{ \includegraphics[width=3cm]{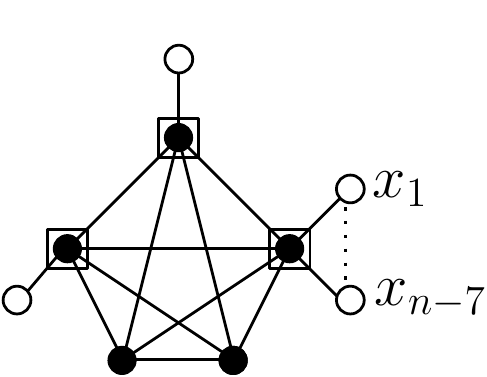} \ \ \ $n\geq 8$}\\
\hline
6 & \includegraphics[width=2cm]{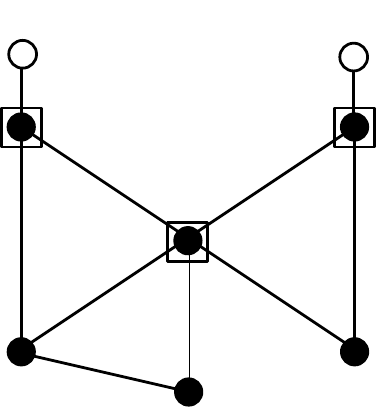} & \multicolumn{3}{l|}{ \includegraphics[width=3cm]{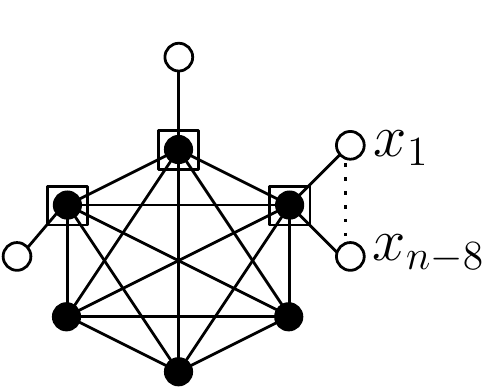} \ \ \ $n\geq 9$}\\
\hline
7 & \includegraphics[width=2.8cm]{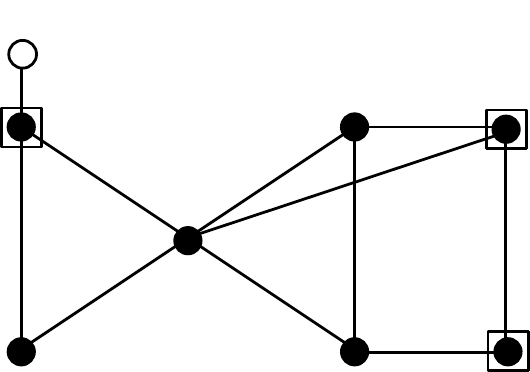}  & \includegraphics[width=2cm]{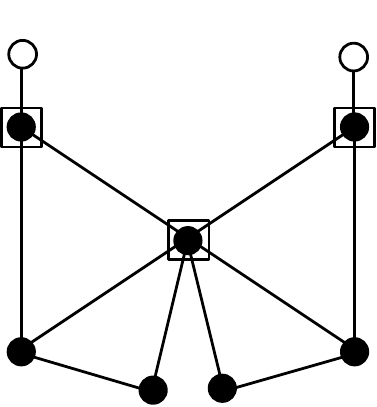} &\multicolumn{2}{l|}{ \includegraphics[width=3cm]{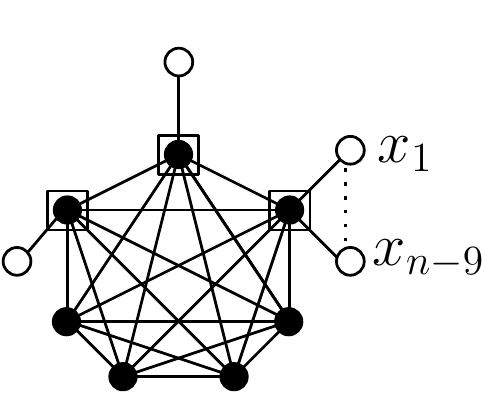} \ \ \ $n\geq 10$}\\
\hline
8 & \pbox{20cm}{\pbox{20cm}{does not exist \\ \\  \\  }} & \includegraphics[width=3cm]{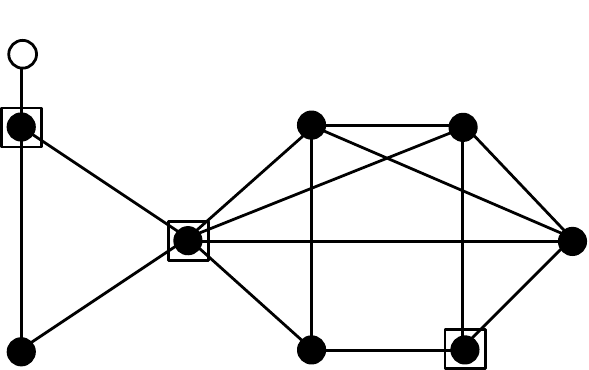} &\includegraphics[width=2cm]{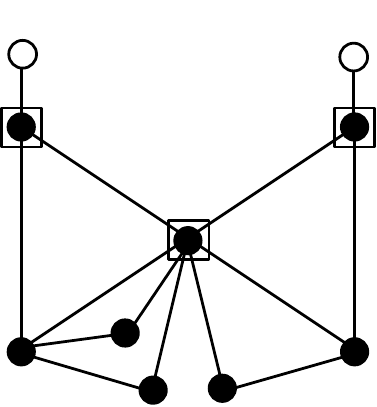}  &\includegraphics[width=3cm]{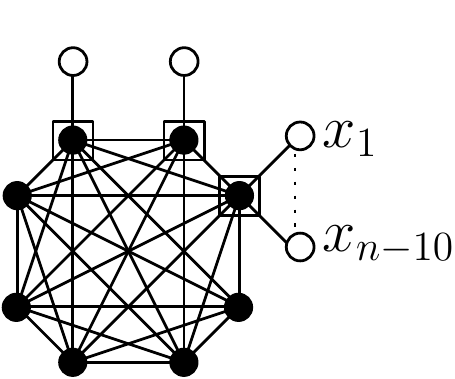} $n\geq 11$\\
\hline
$\geq$9 & \pbox{20cm}{no sense\\  \\  } & \multicolumn{3}{l|}{ \includegraphics[width=5cm]{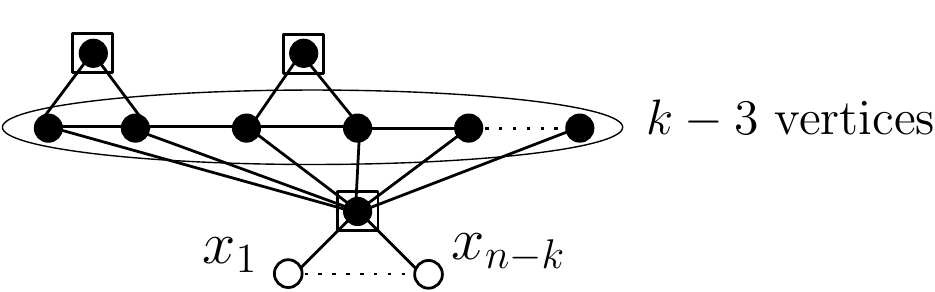} \ \ \ $9\leq \gamma_ {\stackrel{}{11}}=k\leq n$} \\
\hline
\end{tabular}
\caption{Examples of graphs with $n\geq 7$, $\Delta(G)=n-3$, $\gamma=3$ and $\gamma_ {\stackrel{}{11}}\geq 4$. The triplet of squared vertices in each graph is a $\gamma$-code and black vertices form a $\gamma_ {\stackrel{}{11}}$-code. Observe that every $\gamma_ {\stackrel{}{11}}$-code is formed by all non-leaf vertices.}
\label{tablegamma3}
\end{center}
\end{table}

\begin{case}
The cycle $C$ contains exactly two vertices $a$ and $b$ with $\deg(a)=\deg(b)=2$ and they satisfy $d(a,b)=1$.
\end{case}
There are three possible situations:
\begin{enumerate}
\item There exists a vertex $w\in V(G)$ with $\deg(w)=1$. Then it is clear that $w$ is neither a neighbor of $a$ nor of $b$, and so $V(G)\setminus \{a,b,w\}$ is a $\gamma_ {\stackrel{}{11}}$-set.

\item There exists a vertex $w\in V(G)\setminus\{a,b\}$ with $\deg(w)=2$. In this case $w\notin V(C)$ and $d(a,w)\geq 2, d(b,w)\geq 2$. Having in mind that all vertices in $C$, different from $a$ and $b$ have degree 3, going along an induced path from $a$ (or $b$) to $w$, it is possible to find vertices $u,v\in V(G)$ such that $\deg(u)=\deg(v)=2$, $d(u,v)\geq 2$ and all vertices, other than $u$ and $v$, in the induced path between them, have degree 3. So $G$ satisfies condition 2 of Lemma~\ref{conditions}.

\item Any vertex $w\in V(G)\setminus\{a,b\}$ satisfies $\deg(w)=3$. Then there must be another cycle $D$ in $G$, different from $C$, and it is clear that $a,b\notin V(D)$. So $G$ satisfies condition 1 of Lemma~\ref{conditions}.

\end{enumerate}

\begin{case}
The cycle $C$ contains exactly one vertex $a$ with $\deg(a)=2$.
\end{case}
Firstly, if there exists $w\in V(G)\setminus\{a\}$ with $\deg(w)=2$, going along an induced path from $a$ to $w$, we can find vertices $u,v\in V(G)$ that ensure $G$ satisfies condition 2 of Lemma~\ref{conditions}. So suppose now that any vertex $w\in V(G)\setminus\{a\}$ has degree either 1 or 3. If there is another cycle $D$ in $G$, different from $C$, then $G$ satisfies condition 1 of Lemma~\ref{conditions}. Therefore assume that $G$ is an unicyclic graph, with cycle $C$.

\begin{enumerate}
\item If $C$ has at least 4 vertices, from the fact that it has just one vertex of degree 2 it follows that there are three or more vertices of degree 3. Thus $G$ has at least three leaves and therefore $\gamma_{11}(G)\leq n-3$.
\item If $C=C_3$, let $\{a,x,y\}$ be the vertices of $C_3$, where $a$ is the vertex of degree $2$ in $G$. Using that $G$ is not the bull graph (see Figure~\ref{nobull}), it is clear that vertex $z$, the neighbor of $x$ not in $C_3$, is not a leaf so it has degree 3 (remember that we have assumed that no vertex other than $a$ has degree 2). Then there are at least three leaves in $G$  and $\gamma_{11}(G)\leq n-3$.
\end{enumerate}

\begin{case}
All vertices in $C$ have degree 3.
\end{case}
Here $G$ satisfies condition 1 of Lemma~\ref{conditions}.
\end{proof}

\begin{figure}[htbp]
  \begin{centering}
    \subfigure[Bull graph.\label{bull}]{\includegraphics[width=0.12\textwidth]{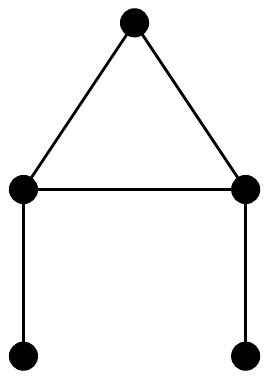}} \hspace{1.5cm}
    \subfigure[Not a bull.\label{nobull}]{\includegraphics[width=0.20\textwidth]{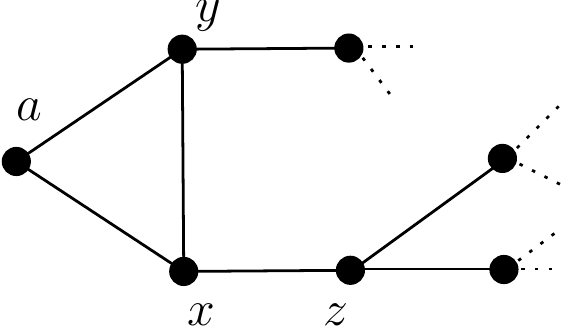}}
\caption{Graphs with $\Delta=3$, containing a cycle $C_3$ and with a unique vertex of degree $2$.}
  \end{centering}
\end{figure}

Note that the upper bound in Theorem~\ref{caseDelta=3} is tight, for instance $\gamma_ {\stackrel{}{11}}(K_{1,3})=1=n-3$. However this bound can be improved in some special cases.

\begin{proposition}
Let $G$ be a cubic graph other than the complete graph with four vertices $K_4$. Then $\gamma_ {\stackrel{}{11}}(G)\leq n-4$.
\end{proposition}

\begin{proof}
Let $G$ be a cubic graphs other than $K_4$. If $G$ has an induced cycle with at least $4$ vertices, then using condition 1 of Lemma~\ref{conditions} we obtain $\gamma_ {\stackrel{}{11}}(G)\leq n-4$. Suppose on the contrary that $G$ contains no induced cycle of length greater of equal than $4$ so $G$ is a chordal graph. It is well known that chordal graphs have a perfect elimination ordering, so we order the vertex set $V(G)=\{v_1,\dots ,v_n\}$ in that way. Then for any vertex, its neighbors occurring after it in the order form a clique. Applying this property to $v_1$ we obtain that its three neighbors form a triangle, so $G=K_4$.
\end{proof}

\begin{proposition}
Let $T$ be a tree with order $n\ge 7$ and $\Delta (G)=3$. Then $\gamma_ {\stackrel{}{11}}(T)\leq n-4$.
\end{proposition}

\begin{proof}
If $T$ has at least two vertices of degree $3$, then it has at least four leaves and we are done. So suppose that there exists an unique vertex $u\in V(T)$ with $deg(u)=3$. We denote by $A,B,C$ the three sets of vertices of the connected components of $T\setminus\{ u\}$, with $|A|\leq |B|\leq |C|$. If $A=\{a\}$ and  $B=\{b\}$ then $C=\{c_1,c_2,c_3,\dots ,c_k\}$ with $k\geq 4$ and we define $S=\{ u,c_3,\dots ,c_k\}$ (see Figure~\ref{tree_71}). If $A=\{a\}$ and  $B=\{b_1,\dots b_r\}$ with $r\geq 2$ then $C=\{c_1,\dots c_k\}$ with $k\geq 3$ and we define $S=\{ a,b_2,\dots ,b_r, c_2,\dots ,c_{k-1}\}$ (see Figure~\ref{tree_72}). Finally if  $A=\{a_1,\dots a_s\}$ with $s\geq 2$ then $B=\{b_1,\dots b_r\}$ with $r\geq 2$ and  $C=\{c_1,\dots c_k\}$ with $k\geq 2$ and we define $S=\{ a_2,\dots, a_s,b_2,\dots ,b_r, c_1,\dots ,c_{k-1}\}$ (see Figure~\ref{tree_73}). In all cases $S$ is a $\gamma_ {\stackrel{}{11}}$-set  of $T$ with $|S|=n-4$.
\end{proof}

\begin{figure}[ht]
  \begin{centering}
    \subfigure[Case 1.\label{tree_71}]{\includegraphics[width=0.2\textwidth]{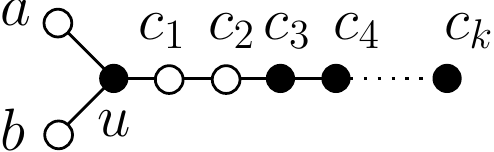}} \hspace{1.5cm}
    \subfigure[Case 2.\label{tree_72}]{\includegraphics[width=0.25\textwidth]{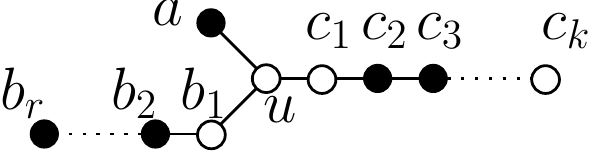}} \hspace{1.5cm}
    \subfigure[Case 3.\label{tree_73}]{\includegraphics[width=0.2\textwidth]{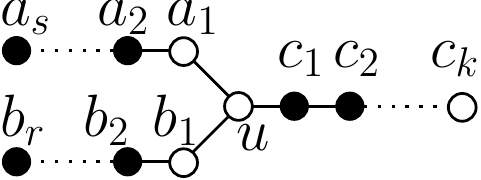}}
\caption{The set of black vertices is a $\gamma_ {\stackrel{}{11}}$-set.}
  \end{centering}
\end{figure}

\section{Cographs and claw-free graphs}\label{freegraphs}
Theorem~\ref{condicion} provides different conditions for a graph to have a short quasiperfect domination chain. Certain extremal values of the maximum degree guarantees that, as we have just studied in Section~\ref{extremal}. However in this section, we are interested in those cases in which the graph belongs to special classes, namely cographs and claw-free graphs.

\subsection{Cographs}\label{cographs}

Cographs are inductively defined as follows~\cite{corneil,bm}:
\begin{itemize}
\item Every single vertex graph is a cograph.
\item If $G_1$ and $G_2$ are two cographs, then their disjoint union is a cograph.
\item The join graph $G_1\vee G_2$ of two cographs is a cograph. Recall that the join graph $G_1\vee G_2$ is obtained from their disjoint union by adding all
edges between vertices of $G_1$ and $G_2$.
\end{itemize}

The next result gives us the values of $\gamma_ {\stackrel{}{11}} (G)$ where $G$ is the join of two graphs.

\begin{theorem}
Let $G=G_1\vee G_2$ be a graph of order $n$. Then,
\begin{enumerate}
\item $\gamma_ {\stackrel{}{11}} (G)=1$ if and only if $G_1$ or $G_2$ have a universal vertex.
\item $\gamma_ {\stackrel{}{11}} (G)=2$ if and only if both $G_1$ and $G_2$ have at least an isolated vertex.
\item $\gamma_ {\stackrel{}{11}} (G)=n$ in other case.
\end{enumerate}
\end{theorem}

\begin{proof}
The first claim is obvious since $\gamma_ {\stackrel{}{11}}(G)=1$ if and only if $\gamma (G)=1$, and a universal vertex of $G$ is a universal vertex either of $G_1$ or $G_2$. From now on, we assume that there is no universal vertex in $G$.

For the second claim, let $u_1\in V(G_1)$ and $u_2\in V(G_2)$ be isolated vertices in $G_1$ and $G_2$ respectively, then it is clear that $\{u_1,u_2\}$ is a $\gamma_ {\stackrel{}{11}}$-code of $G$. Conversely, suppose that $G_1$ (without loss of generality) has no isolated vertex. Let $S$ be a $\gamma_ {\stackrel{}{11}}$-code of $G$, then $S$ must contain at least one vertex $v_1\in V(G_1)$ and one vertex $v_2\in V(G_2)$. Let $x$ be a neighbor of $v_1$ in $G_1$, then it has at least two neighbors in $S$, so $x\in S$ and $\gamma_ {\stackrel{}{11}} (G)\geq 3$.

Finally if $G$ has no universal vertex and $G_1$ has no isolated vertices, we know that $\gamma_ {\stackrel{}{11}} (G)\geq 3$. Let $S$ be a $\gamma_ {\stackrel{}{11}}$-code of $G$, then there are at least two vertices of $S$ in $V(G_1)$ which implies $V(G_2)\subset S$. Note that $|V(G_2)|\geq 2$, because there is no universal vertex in $G$, so there are at least two vertices of $S$ in $V(G_2)$ and also $V(G_1)\subset S$, as desired.
\end{proof}

Note that a connected cograph is the join of two cographs. Thus, the above theorem applies also to those graphs.

\begin{corollary}
Let $G=G_1\vee G_2$ be a connected cograph without universal vertices. Then, $\gamma_ {\stackrel{}{11}} (G)=2$ if both $G_1$ and $G_2$ have at least an isolated vertex, and $\gamma_ {\stackrel{}{11}} (G)=n$ in any other case.
\end{corollary}

\subsection{Claw-free graphs}\label{clawfreeg}

Claw-free graphs, also known as $K_{1,3}$-free graphs, is another graph family where $\gamma=\gamma_ {\stackrel{}{12}}$ according to Theorem~\ref{condicion}. The next result provides examples of claw-free graphs for a great variety of different values for $\gamma, \gamma_ {\stackrel{}{11}}$ and $n$.

\begin{theorem}\label{claw-free th} Let $h,k,n$ be integers such that $4\le n$, $2\le h\le k\le n$ satisfying $h+k\le n$ or $3\, h+k+1\le 2\, n$. Then, there exists a claw-free graph $G$ of order $n$ such that $\gamma (G)=h$ and $\gamma_ {\stackrel{}{11}} (G)=k$.
\end{theorem}

\begin{proof}
First suppose that $h+k\le n$. Let $r=n-(h+k)+2$. We consider the graph $G$ formed by two complete graphs of order $r$ and $k$, sharing exactly one vertex $v$, and $h-1$ vertices of degree 1 pending from distinct vertices $u_1,\dots ,u_{h-1}$ of the complete graph $K_k$ different from $v$ (see Figure \ref{fig.clawfree}(a)).
Note that $G$ is a claw-free graph and since $r\ge 2$, the sets $\{ u_1,\dots ,u_{h-1}, v\}$ and $V(K_k)$ are respectively a $\gamma$-code with $h$ vertices and a $\gamma_ {\stackrel{}{11}}$-code with $k$ vertices.

Suppose now that $3\, h+k+1\le 2\, n$ and $h+k>n$. Let $r=n-h$ and $s=n-k$. Then,

$$2n\geq 3h+k+1\Rightarrow (n-h)+(n-k)\ge 2h+1\Rightarrow r+s > 2 h\Rightarrow  r > 2 h-s=s+2 (h-s)$$
where $s=n-k\ge 0$ and $h-s=h-(n-k)=(k+h)-n\ge 1$.
Therefore it is possible to construct the following graph $G$: consider the graph $K_r$, and $s+2\, (h-s)$ different vertices of $K_r$,  $u_1,\dots ,u_s, v_1,\dots ,v_{h-s},w_1,\dots ,w_{h-s}$. Attach a vertex of degree 1 to each vertex $u_1,\dots ,u_s$ and consider $h-s$ vertices $x_1,\dots ,x_{h-s}$ of degree $2$, where $x_i$ is adjacent to $v_i$ and $w_i$ (see Figure \ref{fig.clawfree}(b)).
Notice that $G$ is a claw-free graph and since $r>s+2 (h-s)$, the sets $\{ u_1,\dots ,u_s,v_1,\dots ,v_{h-s} \}$ and $V(K_r)\cup \{ x_1,\dots ,x_{h-s}\}$ are respectively a $\gamma$-code of $G$ with $h$ vertices and a $\gamma_ {\stackrel{}{11}}$-code of $G$ with $n-s=k$ vertices.
\end{proof}

\begin{figure}[!hbt]
\begin{center}
\includegraphics[width=0.6\textwidth]{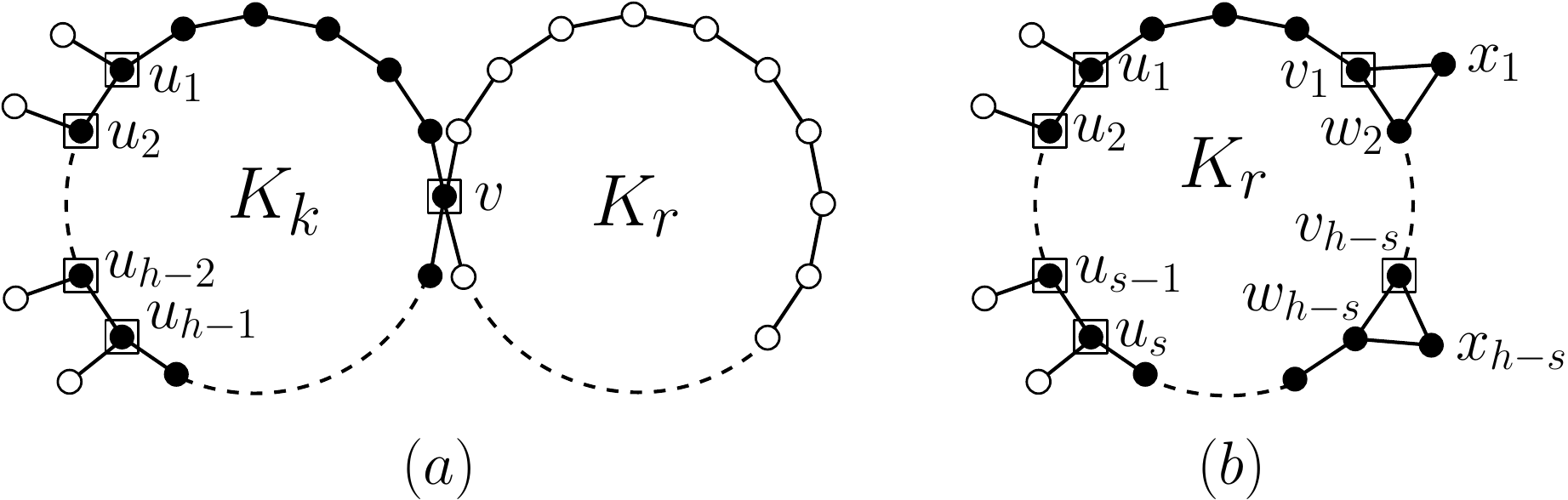}
\caption{Claw-free graphs on the proof of Theorem~\ref{claw-free th}. The set of squared vertices is a $\gamma$-code and black vertices form a $\gamma_ {\stackrel{}{11}}$-code.}\label{fig.clawfree}
\end{center}
\end{figure}

Conditions $h+k\le n$ or $3\, h+k+1\le 2\, n$ in Theorem~\ref{claw-free th} are sufficient to ensure that there exists a claw-free graph $G$ of order $n$ such that $\gamma (G)=h$ and $\gamma_ {\stackrel{}{11}} (G)=k$. There are some cases where the reverse is also true. For instance, if $G$ is a claw-free graph with $\gamma (G)=\frac{n}{2}$, $n$ even, then $G$ is the cycle $C_4$ or $G$ is the corona graph of a complete graph $K_m$ (see \cite{BaCoHa}) and $k=\gamma_ {\stackrel{}{11}}(G)=\frac{n}{2}$, so $h+k\leq n$. Also in the following proposition we show that they are necessary conditions, with just two exceptions, in the case of graphs with small order. So we think that the reverse of Theorem~\ref{claw-free th} could be true in a wider range of cases.

\begin{proposition}\label{small cases}
Let $h,k,n$ be integers such that $4\le n\le 7$  and $2 \le h \le k\le n$. Then, there exists a claw-free graph $G$ of order $n$ such that $\gamma(G)=h$ and $\gamma_ {\stackrel{}{11}}(G)=k$ if and only if $h+k\le n$ or $3h+k+1\le 2n$ or $(h,k,n)=(2,6,6)$.
\end{proposition}
\begin{proof}

Firstly, if $h,k,n$ satisfy hypothesis then, using Theorem~\ref{claw-free th} we obtain the desired graphs, except in case $(h,k,n)=(2,6,6)$, that is shown in Figure~\ref{8a}.

Conversely suppose that $G$ is a claw-free graph with order $n$ and such that $\gamma(G)=h$ (so $h \le \frac{n}{2}$) and $\gamma_ {\stackrel{}{11}}(G)=k$ with $4\le n\le 7$.
If $\Delta (G)=n-1$ then $h=k=1$, which is not our case. If $\Delta(G)=2$, then $G$ must be the n-cycle or the n-path, with $4\le n\le 7$, and it is easy to check that $h+k\le n$ in all cases. This completely solves the case $n=4$. In the remaining cases we classify graphs using the maximum degree $\Delta(G)$ where $3\le \Delta(G)\leq n-2$.

If $n=5$ then $h=2$. The only case we have to check is $\Delta(G)=3$, so $G$ is the bull graph that satisfies $h=2$, $k=3$ or $G$ is not the bull graph and using Theorem~\ref{caseDelta=3}, $h=2$, $k=n-3=2$. In both cases $h+k\le n$.

If $n=6$ then $2\le h \le 3$. If $\Delta (G)=3$, Theorem~\ref{caseDelta=3} gives $k\le n-3=3$ that implies $h+k\le n$. If $\Delta(G)=4$ then $h=2$ and we distinguish to options: $k=6$ implies $(h,k,n)=(2,6,6)$ and $k\le 5$ means $3h+k+1\le 3\cdot 2+5+1=12=2\,n$.

If $n=7$ then again $2\le h \le 3$. In the case $\Delta(G)=3$, using Theorem~\ref{caseDelta=3}, we obtain $k\le n-3=4$ so $h+k\le n$. If $\Delta (G)=4$ and $h=2$ then $3\, h+k+1\le 3\cdot 2+7+1=2\,n$. On the other hand it is easy to check (\cite{sage}) that there are exactly three claw-free graphs of order $7$ with $\Delta (G)=4$ and $h=3$ (see Figures~\ref{cf7_1},~\ref{cf7_2},~\ref{cf7_3}) and they satisfy $3\le k\le 4$, so $h+k\le n$. Finally $\Delta(G)=5$ implies $h=2$ and $3\, h+k+1\le 3\cdot 2+7+1=2\,n$.
\end{proof}

\begin{figure}[htbp]
  \begin{centering}
  \subfigure[$n=6$, $h=2$, \newline $k=6$.\label{8a}]{ \includegraphics[width=0.15\textwidth]{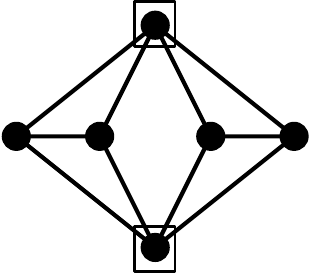}} \hspace{1.5cm}
  \subfigure[$n=7$, $\Delta=4$, \newline $h=3$, $k=4$.\label{cf7_1}]{ \includegraphics[width=0.15\textwidth]{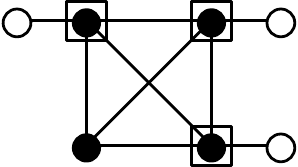}} \hspace{1.5cm}
  \subfigure[$n=7$, $\Delta=4$, \newline $h=3$, $k=3$.\label{cf7_2}]{ \includegraphics[width=0.15\textwidth]{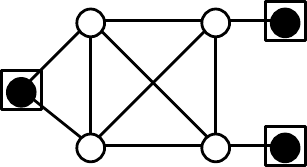}} \hspace{1.5cm}
\subfigure[$n=7$, $\Delta=4$, \newline $h=3$, $k=3$.\label{cf7_3}]{ \includegraphics[width=0.15\textwidth]{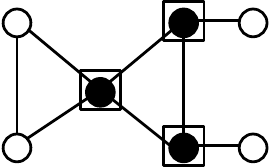}} \hspace{1.5cm}
\caption{Squared vertices are a $\gamma$-code and black vertices are a $\gamma_ {\stackrel{}{11}}$-code.}\label{cfs}
  \end{centering}
\end{figure}

\section*{Acknowledgements}
The authors would like to thank professor Carlos Seara for his continuous support and feedback during the realization of this paper.

\end{document}